\documentclass[12pt]{amsart}

\usepackage{setspace}
\usepackage{amsmath, amssymb,amsthm, amsfonts}
\newtheorem{thm}{Theorem}
\newtheorem{lma}{Lemma}
\newtheorem{prop}{Proposition}
\newtheorem{cor}{Corollary}

\theoremstyle{definition}

\theoremstyle{remark}
\newtheorem{remark}{Remark}

\newcommand{\R}{\mathbb R}

\newcommand{\be}{\begin{equation}}
\newcommand{\ee}{\end{equation}}
\newcommand{\bee}{\begin{equation*}}
\newcommand{\eee}{\end{equation*}}

\def\p{\partial}

\def\lf{\left}
\def\ri{\right}
\def\Pi{\displaystyle{\mathbb{II}}}

\def\a{\alpha}

\def\ccs{c_{\rm CS}}
\def\bccs{\mathbf{  c}_{\rm CS}}

\def\bchy{\mathbf{  c}_{\rm HY}}

\begin{document}

\title[]{A note   on center of mass}

\author{Pak-Yeung Chan}
\address{Department of Mathematics,
The Chinese University of Hong Kong} \email{pychan@math.cuhk.edu.hk}

\author{Luen-Fai Tam$^1$}

\thanks{$^1$Research partially supported by Hong Kong RGC General Research Fund  \#CUHK 403108}

\address{The Institute of Mathematical Sciences and Department of
 Mathematics, The Chinese University of Hong Kong,
Shatin, Hong Kong, China.} \email{lftam@math.cuhk.edu.hk}

\begin{abstract} We will discuss existence of center of mass on asymptotically Schwarzschild manifold defined by Huisken-Yau \cite{Huisken-Yau1996} and Corvino-Schoen \cite{CorvinoSchoen2006}.
Conditions of existence and   examples on non existence are given.

\end{abstract}
\date{February, 2014}
\maketitle
\markboth{Pak-Yeung Chan and Luen-Fai Tam} {A note on center of mass}

\S{\bf 1.} Let $(M^3,g)$ be an asymptotically Schwarzschild (AS) manifold. That is: $M$ is diffeomorphic to $\R^3\setminus B(R)$ with metric $g$   given by
\be\label{eq-AS-1}
g_{ij}=\lf(1+\frac{m}{2r}\ri)^4\delta_{ij}+p_{ij}
\ee
where $p_{ij}=O_4(r^{-2})$, where $m>0$ is a constant  is the ADM mass of the manifold. Here $r=|x|$. The notation $\phi=O_k(r^\a)$ means that $|\p^{(i)}\phi|\le Cr^{\a-i}$ for some constant $C$ for all $0\le i\le k$. We assume $g$ is extended smoothly to the whole $\R^3$.

In \cite{Huisken-Yau1996}, Huisken-Yau proved the existence and uniqueness of constant mean curvature stable foliation $\{\Sigma_r\}$ which are perturbation of the coordinate spheres. Let $F(r)$ be the embedding of $\Sigma_r$ in $M$. The Husiken-Yau center of mass is defined as follow:
Let
\be\label{eq-HY-1}
\bchy(r)=\frac{\int_{\Sigma_r}F(r) d\sigma_0}{\int_{\Sigma_r} d\sigma_0}
\ee
where $d\sigma_0$ is the area element induced by the Euclidean metric. The  Husiken-Yau center of mass $\bchy$ is defined as:
\be\label{eq-HY-2}
\bchy=\lim_{r\to\infty}\bchy(r)
\ee
provided the limit exists.

In \cite{CorvinoSchoen2006} there is another definition of center of mass defined by   Corvino-Schoen. Let
 \be\label{eq-CS-1}
\begin{split}
\ccs^\a(r)=&\frac1{16\pi m}\int_{|x|=r}x^\a \lf[ (g_{ij,i}-g_{ii,j})\nu_g^j-\lf(h_{i\a}\nu_g^i-h_{ii}\nu_g^\a\ri)\ri]d\sigma_g\\
\end{split}
\ee
where $h_{ij}=g_{ij}-\delta_{ij}$ and $\nu_g$ is the unit outward normal of $\{|x|=r\}$ with respect to $g$.
Let $\bccs(r)=(\ccs^1(r),\ccs^2(r),\ccs^3(r))$. The Corvino-Schoen center of mass is given by
\be\label{eq-CS-2}
\bccs=\lim_{r\to\infty}\bccs(r)
\ee
provided the limit exists. Note that
\be
\begin{split}
\ccs^\a(r)=&\frac1{16\pi m}\int_{|x|=r}x^\a \lf[ (g_{ij,i}-g_{ii,j})\nu_0^j-\lf(h_{i\a}\nu_0^i-h_{ii}\nu_0^\a\ri)\ri]d\sigma_0+O(r^{-1}).\\
\end{split}
\ee
where $\nu_0$ is the unit outward normal of $\{|x|=r\}$ with respect to Euclidean metric.

In this note we want to discuss the existence of $\bchy$ and $\bccs$. By the result of Huang \cite{Huang2009}, the two concepts are basically the same, see also \cite{CorvinoWu2008}. Before we state the precise statement of the result, we will use the foliation constructed by Ye \cite{Ye1996}   which is the same as that by Husiken-Yau near infinity by uniqueness. The foliation constructed by Ye is as follows. For $r>0$ large enough, we can find a perturbed center $\tau(r)\in \R^3$ and a function $\phi^{(r)}(z) $ on the unit sphere $\mathbb{S}^2$ such that the surface
\be\label{eq-ye-foliation}
\Sigma_r=\{r\lf(z+\tau(r)+\phi^{(r)}(z)\nu_0(z)\ri)|\ z\in \mathbb{S}^2\}
\ee
has constant mean curvature $\frac 2r-\frac{4m}{r^2}$. Here $\nu_0$ is the unit outward normal of unit sphere $\mathbb{S}^2$ in $\R^3$. Note that by \cite{Ye1996}, $|\tau(r)|\le Cr^{-1}$ and $|\phi^{(r)}|\le Cr^{-2}$. Define $\bchy(r)$ as in \eqref{eq-CS-1}. Then $\bchy$ is $\lim_{r\to\infty}\bchy(r)$, provided it exists. Huang \cite{Huang2009} proved the following:

\begin{prop}\label{Huang}
$$
\lim_{r\to\infty}\lf(\bccs(r)-\bchy(r)\ri)=\mathbf{0}.
$$
\end{prop}
\begin{proof} We sketch the proof here. Let   $y=x-\frac{\tau}r$, and $y=\frac{z}r$, $z\in \mathbb{S}^2$. So $x=\frac1r(z+\tau)$. Let $\Sigma_r$ be as in \eqref{eq-ye-foliation}. Then one can check that,  using the fact that $|\phi^{(r)}|=O(r^{-1})$, one can check that
\be\label{eq-Huang-1}
\lim_{r\to\infty} (r\tau(r)-\bchy(r))=\mathbf{0}
\ee
On the other hand, by \cite[(1.14)]{Ye1996}, for $\a=1,2,3$, $\tau(r)$ satisfies:
\be\label{eq-Huang-2}
6mr\tau^\a+  P_\a\lf(rf(r,z,\tau)+rb_{ij}(z,\tau)\tau^i\tau^j+w \ri)=0
\ee
where $b_{ij}$ is smooth in $(z,\tau)$, $z\in \mathbb{S}^2$, $|w|=O(r^{-1})$, $P_\a$ is the $L^2$ projection of a function on $ \mathbb{S}^2$ to the linear space spanned by $z^\a$ and $f$ is given by
\be\label{eq-Huang-3}
H(r,\tau(r),0)=\frac2r+\frac {4m^2}{r^2}+ \frac{6m z\cdot \tau}{r^2}+\frac1{r^2}f(r,z,\tau(r))+O(r^{-4})
\ee
and $H(r,\tau,0)$ is the mean curvature of the surface$\{|x-  r\tau  |=r\}$. Let $y=x-r\tau$, $z=\frac y r$. Then $H(r,\tau,0)$  is given by (see \cite[(6.1)]{Huang2009})
   \be
   \begin{split}
   H(r,\tau(r),0)=&\frac2{r}-\frac{4m}{r^2}+\frac{6m z\cdot \tau}{r^2}+\frac{9m^2}{r^3}\\
   &+\frac1{2r^3}q_{ij,k}(y) y^iy^jy^k+\frac{2}{r^3}q_{ij}(y)y^iy^j\\
   &-\frac1 r \lf(q_{ij,i}(y)y^j-q_{ii}(y)+\frac12 q_{ii,j}(y)y^j\ri)+E\\
   \end{split}
   \ee
   where $E=O(r^{-4})$,  $q_{ij}=p_{ij}+\lf(1+\frac{m}{2r}\ri)^4\delta_{ij}-\lf(1+\frac{2m}r\ri)\delta_{ij}$.

Hence by \cite[Lemma 6.1]{Huang2009}
   \be\begin{split}
   P_\a(rf)=&\frac {3\pi}4
   \int_{|z|=1}z^\a rf  d\sigma_0\\
   =&\int_{|z|=1}z^\a r^3\lf(H(r,\tau(r),0)-\frac2r-\frac {4m^2}{r^2}-\frac{6m   z\cdot \tau}{r^2}\ri)d\sigma_0+O(r^{-1})\\
   =&-6m \, \ccs^\a(r)+O(r^{-1}).
   \end{split}
   \ee
   Combining with \eqref{eq-Huang-1} and \eqref{eq-Huang-2}, the result follows.

\end{proof}

   Next we will give condition so that $\bccs$ and hence $\bchy$ exists. The following result is a direct consequence of the computation in \cite[section 5]{Corvino2000} by Corvino and
   \cite[p. 215]{CorvinoSchoen2006} by Corvino-Schoen. However, we would like to state the result explicitly.
\begin{thm}\label{thm-CS}
$\bccs$ exists if and only if
$
\lim_{r\to\infty}\int_{B(r)} x^\a R_g dv_g$ exists for $\a=1, 2, 3$.
\end{thm}
\begin{proof}
   $$
R_{ij}=\p_k\Gamma^k_{ji}-\p_j\Gamma^k_{ki}+
\Gamma^k_{kl}\Gamma^l_{ji}-\Gamma^k_{jl}\Gamma^l_{ki}.
$$
On an AS manifold,
$$
g_{ij}=\lf(1+\frac { m}{2r}\ri)^4\delta_{ij}+p_{ij}
$$
with $p_{ij}=O_4(r^{-2})$. Let
$$
\bar g_{ij}=\lf(1+\frac { m}{2r}\ri)^4\delta_{ij}
$$
and let $\Gamma_{ij}^k,  \bar \Gamma_{ij}^k$ be the Christoffel symbols for $g$ and  $\bar g$ respectively. Extend $(1+\frac {2m}r)$ as a positive function up to the origin, and denote it by $u$.
Then by \cite{Huisken-Yau1996} and direct computations, we have
\be\label{eq-curvature-0}
\begin{split}
\Gamma_{ij}^k-\bar \Gamma_{ij}^k=&\frac12 \bar g^{sk}\lf(p_{is,j}+p_{sj,i}-p_{ij,s}\ri)+\frac12
\lf(g^{sk}-\bar g^{sk}\ri)\Gamma_{ij}^k\\
=&\frac12 \lf(p_{ik,j}+p_{kj,i}-p_{ij,k}\ri)+O(r^{-4})
\end{split}
\ee
Hence
\be\label{eq-curvature-1}
|\Gamma_{ij}^k-\bar \Gamma_{ij}^k|=O(r^{-3}),|\p (\Gamma_{ij}^k-\bar \Gamma_{ij}^k)|+ |R_{ij}-\bar R_{ij}|=O(r^{-4}).
\ee
 In particular, $|R_g|=O(r^{-4})$ because the scalar curvature $R_{\bar g}$ of $\bar g$ is 0 near infinity. Let $dv_0$ be the Euclidean volume element.
\bee
\begin{split}
\int_{B(R)}x^\a R_gdv_g=&\int_{B(R)} x^\a R_gdv_0 +\int_{B(R)}Edv_0\\
=&\int_{B(R)} x^\a g^{ij}R_{ij}dv_0 +\int_{B(R)}Edv_0\\
=&\int_{B(R)} x^\a \lf(g^{ij}R_{ij}-\bar g^{ij}\bar R_{ij}\ri)dv_0 +C+\int_{B(R)}Edv_0\\
=&\int_{B(R)} x^\a \bar g^{ij}\lf( R_{ij}- \bar R_{ij}\ri)dv_0+C+\int_{B(R)}Edv_0\\
=&\int_{B(R)} x^\a u^4(r)   \sum_i(R_{ii}- \bar R_{ii})dv_0+C+\int_{B(R)}Edv_0
\end{split}
\eee
if $R$ is large, where $C$ is a constant independent of $R$. Here and below $E$ always denote a function with $E=O(r^{-4})$. Now
\bee
\lf(\Gamma^k_{kl}\Gamma^l_{ji}-\Gamma^k_{jl}\Gamma^l_{ki}\ri)-\lf(\bar\Gamma^k_{kl}\bar\Gamma^l_{ji}-
\bar\Gamma^k_{jl}\bar\Gamma^l_{ki}\ri)=O(r^{-5}).
\eee
Hence
\be
\begin{split}
&\int_{B(R)}x^\a R_gdv_g\\=&\int_{B(R)} x^\a   \sum_i\lf[\p_k\lf(\Gamma^k_{ii}-\bar \Gamma^k_{ii}\ri)-\p_i\lf(\Gamma^k_{ki}-\bar \Gamma^k_{ki}\ri)\ri] dv_0+C+\int_{B(R)}Edv_0\\
=&\int_{\p B(R)}x^\a\lf[  \sum_{i,k}(\Gamma^k_{ii}-\bar \Gamma^k_{ii}) \nu_0^k- \sum_{i,k}(\Gamma^k_{ki}-\bar \Gamma^k_{ki})\nu_0^i\ri]\\
&-\int_{ B(R)} \lf[  \sum_i(\Gamma^\a_{ii}-\bar \Gamma^\a_{ii})  - \sum_k (\Gamma^k_{k\a}-\bar \Gamma^k_{k\a})\ri]dv_0+C+\int_{B(R)}Edv_0\\
=& \int_{\p B(R)}x^\a\sum_{i,k}\lf(p_{ik,i}-p_{ii,k}\ri)\nu_0^k
-\int_{ B(R)}\sum_{i} \lf(p_{i\a,i}-p_{ii,\a}\ri)dv_0+C+\int_{B(R)}Edv_0\\
=&\int_{\p B(R)}x^\a\lf[\sum_{i,k}\lf(p_{ik,i}-p_{ii,k}\ri)\nu_0^k-\int_{\p B(R)} \sum_{i}\lf(p_{i\a}\nu_0^i-p_{ii}\nu_0^\a\ri) \ri]dv_0+C+\int_{B(R)}Edv_0\\
=&\int_{\p B(R)}x^\a\lf[\sum_{i,k}\lf(g_{ik,i}-g_{ii,k}\ri)\nu_0^k-\int_{\p B(R)} \sum_{i}\lf(h_{i\a}\nu_0^i-h_{ii}\nu_0^\a\ri) \ri]+C+\int_{B(R)}Edv_0,
\end{split}
\ee
where $C$ is a constant and $E=O(r^{-4})$. From this it is easy to see the theorem is true.
\end{proof}

As remark by Huang \cite{Huang}, the result is still true for asymptotically flat metric satisfying Regge-Teitelboim parity condition.

By the theorem, one may expect there are examples of AS metric so that $\bccs$ and hence $\bchy$ does not exist. In fact, one may construct such examples in a more direct way.  To motivate the construction, let $\mathbf{b}$ be a nonzero vector in $\R^3$ and let $g$ be the metric given by

 \be\label{eq-AS-2}
g_{ij}=\lf(1+\frac{m}{2r}+\frac{\mathbf{b}\cdot \mathbf{x}}{r^3}\ri)^4\delta_{ij}
\ee
with $m>0$.
 Then it is well-known that the Corvino-Schoen center of mass for this metric is given by
 $$
 \bccs=\frac{2\mathbf{b}}m.
 $$
 Let $\phi:[a,\infty)\to \R$ be a smooth bounded function. Consider the metric
  \be\label{eq-AS-3}
g_{ij}=\lf(1+\frac{m}{2r}+\frac{\phi(r)\mathbf{b}\cdot \mathbf{x}}{r^3}\ri)^4\delta_{ij}
\ee
with $m>0$. If $\phi(t)$ is oscillating near infinity, then one may expect that Corvino-Schoen center of mass does not exist. More precisely, we have the following:
\begin{thm}\label{t-example} Let  $\phi:[a,\infty)\to \R$ be a smooth   function, for some $a$. Suppose $\phi$ is such that
\be\label{eq-example-1}
|\phi^{(l)}|\le \frac{C}{(1+t)^l}
\ee
for some constant $C$ for $0\le 1\le 4$. Then the metric given by \eqref{eq-AS-3} is AS outside $B(R)$ for some $R>0$. Moreover, if $\mathbf{b}\neq\mathbf{0}$, then the Corvino-Schoen center of mass exists if and only if
$
\lim_{t\to\infty}\lf(3\phi(t)- t\phi'(t)\ri)
$
exists. If $\lim_{t\to\infty}\lf(3\phi(t)- t\phi'(t)\ri)=\lambda$ exists, then
$$
\bccs=\frac{2\lambda\mathbf{b}}{3m}.
$$
\end{thm}

\begin{remark} It is easy to construct $\phi$ satisfying \eqref{eq-example-1}, but the limit $\lim_{t\to\infty}\lf(3\phi(t)- t\phi'(t)\ri)$ does not exist. For example, we may take $\phi(t)= \sin(\log(t))$ or $\phi(t)=\sin(\log(\log(t))$. Note that similar examples for the nonexistence of center of mass have already been obtained independently by Cederbaum and   Nerz \cite[p.13]{CederbaumNerz2013}. We thank Cederbaum and Nerz for the information.
\end{remark}

\begin{proof}[Proof of Theorem \ref{t-example}] To simplify the notations, let
$$
v=\frac{\phi(r)\mathbf{b}\cdot \mathbf{x}}{r^3}
$$
and
$$
u=1+\frac r{2m}+v.
$$
Then $g_{ij}=u^4\delta_{ij}$. Now $|v|=O(r^{-2})$, and
\be
\frac{\p}{\p x^k}v=
 \frac{1}{r^3}\lf( \frac{x^k}r\phi'(r)\mathbf{b}\cdot \mathbf{x}+\phi(r)b^k - \frac{3x^k\phi(r)\mathbf{b}\cdot \mathbf{x}}{r^2}\ri).
\ee
By the assumption \eqref{eq-example-1}, we have $|\p v|=O(r^{-3})$. Similarly, one can prove that $|\p^2v|=O(r^{-4})$, $|\p^3v|=O(r^{-5})$, $|\p^4v|=O(r^{-6})$. From these, one can see that the metric $g$ is well-defined and is AS.

Next, we want to compute $\ccs^\a(r)$.
\bee
\begin{split}
g_{ij,k}=&4u^3\frac{\p u}{\p x^k}\delta_{ij}\\
=&4u^3\lf[-\frac{ mx^k}{2r^3}+ \frac{1}{r^3}\lf( \frac{x^k}r\phi'(r)\mathbf{b}\cdot \mathbf{x}+\phi(r)b^k - \frac{3x^k\phi(r)\mathbf{b}\cdot \mathbf{x}}{r^2}\ri) \ri]\delta_{ij}\\
=&f_k\delta_{ij}.
\end{split}
\eee
Hence
\be\label{eq-example-3}
\begin{split}
&\sum_{i,j}(g_{ij,i}-g_{ii,j}) x^j\\
=&-2\sum_j f_jx_j\sum_{i,j}(f_i\delta_{ij} -f_j\delta_{ii}) x^j\\
=&\sum_jf_jx^j-3\sum_j f_jx^j\\
=&-8u^3\sum_j x^j\lf[-\frac{ mx^j}{2r^3}+ \frac1{r^3}\lf( \frac{x^j}r\phi'(r)\mathbf{b}\cdot \mathbf{x}+\phi(r)b^j - \frac{3x^j\phi(r)\mathbf{b}\cdot \mathbf{x}}{r^2}\ri) \ri] \\
=&-8u^3\lf[-\frac{ m }{2r}+\frac{(r \phi'(r)-2\phi(r))\mathbf{b}\cdot \mathbf{x}} {r^3}\ri]\\
=&-8(1+\frac{3m}{2r})\lf[-\frac{ m }{2r}+\frac{(r \phi'(r)-2\phi(r))\mathbf{b}\cdot \mathbf{x}} {r^3}\ri]+O(r^{-3})\\
=&-8\lf[-\frac{ m }{2r}-\frac{3m^2}{4r^2}+\frac{(r \phi'(r)-2\phi(r))\mathbf{b}\cdot \mathbf{x}} {r^3}\ri]+O(r^{-3})\\
=& \frac{4 m }{r}+\frac{6m^2}{ r^2}-\frac{8(r\phi'(r)-2\phi(r))\mathbf{b}\cdot \mathbf{x}} {r^3} +O(r^{-3})
\end{split}
\ee
On the other hand, $h_{ij}=(u^4-1)\delta_{ij}$
Hence
\be
\begin{split}
\sum_ih_{i\a} x^i -h_{ii} x^\a=&-2(u^4-1) x^\a\\
=&-2x^\a\lf(\dfrac{ 2 m }{ r}+ \frac{3m^2}{2r^2}+\frac{4\phi(r)\mathbf{b}\cdot \mathbf{x}}{r^3}\ri)+O(r^{-3})\\
\end{split}
\ee

So
\be
\begin{split}
 x^\a\sum_{i,j}(g_{ij,i}-g_{ii,j}) x^j&-\lf(\sum_ih_{i1} x^i -h_{ii} x^1\ri)\\
= & x^\a\lf[\frac{8m}r+\frac{9m^2}{r^2}+\frac{8(3\phi(r)-r\phi'(r))\mathbf{b}\cdot \mathbf{x}} {r^3}\ri]+O(r^{-3})
\end{split}
\ee
Hence
\be
\frac1r\int_{|x|=r}x^1 (g_{ij,i}-g_{ii,j}) x^j -\lf(h_{i\a} x^i -h_{ii} x^1\ri)= \frac{32\pi b^\a}3\lf[3\phi(r)- r\phi'(r)\ri]+O(r^{-1})
\ee
$$
\ccs^\a(r)=\frac{2 b^\a}3\lf[3\phi(r)- r\phi'(r)\ri]+O(r^{-1}).
$$
From this the result follows.
\end{proof}
\begin{remark} (i) If $m<0$, the result is still true if we use the foliation of Ye \cite{Ye1996} to define the center of mass as in \eqref{eq-HY-1} and \eqref{eq-HY-2}.

(ii)  One can check the examples in the theorem satisfy  the property that $ \bccs(r) $ remain bounded for all $r$. On the other hand, in \cite{Huang2010}, Huang constructed examples of asymptotically flat manifold so that $\bccs(r)\to\infty$.
\end{remark}

\S{\bf 2}. The examples constructed in the previous section are very simple. However, there is a drawback. In fact, after the first draft of this work, Huang \cite{Huang} asked whether there is an example with nonnegative scalar curvature. Wang \cite{Wang} also pointed out that the above examples do not have nonnegative scalar curvature.   In this section, we give an affirmative answer to Huang's question. We will construct examples with nonnegative scalar curvature, asymptotically Schwarzschild, so that the Huisken-Yau center of mass and Corvino-Schoen center of mass do not exist. Note that by \cite{Huisken-Yau1996}, the scalar curvature must decay like $r^{-4}$. Given a function $f$ with this decay rate, it is not so difficult to construct asymptotically flat manifold with scalar curvature being $af$ for some positive constant $a$. However, in order to obtain asymptotically Schwarzschild metric, we need an addition assumption on $f$. We begin with the following:

\begin{lma}\label{l-existence-1} Let $f$ be a smooth function on $\R^3$ satisfying the following:
\begin{enumerate}
  \item [(i)] $f=O_3( |x|^{-4})$.
  \item [(ii)] There is a constant $C>0$ such that
  \be
  \lf|\int_{B(r)}x^\a f(x) dv_0\ri|\le C
  \ee
  for $\a=1, 2, 3$ and for all $r>0$, where $B(r)=B_0(r)$ is the Euclidean ball with center at the origin and with radius $r$.
\end{enumerate}
Then there is a smooth function $v$ such that $\Delta v=f$. Near infinity $v(x)=\frac{m}{2|x|}+w(x)$ for some constant $m$ and $w=O_4(|x|^{-2}).$
\end{lma}
\begin{proof} Let $v(x)$ be such that
$$
4\pi v(x)=-\int_{\R^3}\frac1{|x-y|}f(y)dv_0(y).
$$
Then $v$ is well-defined because of (i) and $\Delta v=f$. We want find the asymptotically behavior of $v$. For any $x\in \R^3$, let $r=|x|$. Suppose $r>1$.

\be
\begin{split}
\int_{\R^3}\frac1{|x-y|}f(y)dv_0(y)=&\lf(\int_{B_x(\frac r2)}+\int_{B_0(\frac r2)}+\int_{\R^3\setminus (B_x(\frac r2)\cup B_0(\frac r2))}\ri) \frac1{|x-y|}f(y)dv_0(y)\\
=&I+II+III.
\end{split}
\ee
Now
\be
|I|\le C_1 r^{-4}\int_{B_x(\frac r2)}\frac1{|x-y|} dv_0(y)\le C_2r^2.
\ee
Here and below, $C_i$ will denote positive constant which is independent of $x$ and $r$.

Since outside $B_x(\frac r2)$, $|x-y|\ge \frac r2$. Hence
\be
|III|\le C_3 r^{-1}\int_{\R^3\setminus   B_0(\frac r2) }|y|^{-4}dv_0(y)\le C_4r^{-2}.
\ee

To estimate $II$, let $a=  \int_{\R^3}fdv_0$. Note that $a$ is a finite number because of (i). For $r>1$,
\be
II-\frac{a}{r}=-\frac1r\int_{\R^3\setminus B_0(\frac r2)} fdv_0+ \int_{B_0(\frac r2)}\lf(\frac1{|x-y|}-\frac1{|x|}\ri)f(y)dv_0(y)=IV+V.
\ee
Now $|IV|\le C_5 r^{-2}$. Since $|x-y|\ge \frac r2$ for $y\in B_0(\frac r2)$, we have
\be
\begin{split}
  |V|=&\lf|\int_{B_0(\frac r2)}\lf(\frac{ 2x\cdot y-|y|^2}{|x|\,|x-y|\,(|x|+|x-y|)}\ri)f(y)dv_0(y)\ri|\\
\le &\lf|\int_{B_0(\frac r2)}\lf(\frac{ 2x\cdot y }{|x|\,|x-y|\,(|x|+|x-y|)}\ri)f(y)dv_0(y)\ri|+C_6r^{-2}
\end{split}
\ee
Now for $y\in B_0(\frac r2)$,
$$
\lf|\frac1{|x|}-\frac1{|x-y|}\ri|\le C_7 r^{-2}|y|,
$$
and
$$
\lf|\frac1{2|x|}-\frac1{|x|+|x+y|}\ri|n\le C_8r^{-2}|y|.
$$
So
\be
\begin{split}
\bigg|\int_{B_0(\frac r2)}&\lf(\frac{ 2x\cdot y }{|x|\,|x-y|\,(|x|+|x-y|)}\ri)f(y)dv_0(y)\bigg|\\
\le& C_9r^{-2}+\sum_\a\frac {|x^\a|}{r^3}\lf|\int_{B_0(\frac r2)}y^\a f(y)dv_0(y)\ri|\\
\le & C_{10} r^{-2}
\end{split}
\ee
by assumption (ii). Hence one can see that near infinity
$$
v(x)=\frac{m}{2|x|}+ w(x)
$$
for some smooth $w(x)$ so that $|w(x)|=O(|x|^{-2})$. Since $\Delta \frac1{|x|}=0$, we still have
$$
\Delta w=f
$$
near infinity. By interior Schauder estimate \cite[Theorem 6.2 and problem 6.1]{GilbargTrudinger}, and assumption (i), we conclude that the lemma is true.
\end{proof}
\begin{remark}\label{rem-decay} From the computation in the lemma, one can see that if we only assume that $|f(x)|\le C|x|^{-4}$, then we can only  have the following estimate of $v$.
$$
v(x)=\frac{m}{2|x|}+O(|x|^{-2}\log |x|).
$$
In any case, we still have $v(x)=O(|x|^{-1})$.
\end{remark}
\begin{thm}\label{t-existence} Let $K(x)$ be a smooth function on $\R^3$ satisfying the following:
\begin{enumerate}
  \item [(i)] $K(x)=O_3( |x|^{-4})$.
  \item [(ii)]  There is a constant $C>0$ such that
  \be
  \lf|\int_{B(r)}x^\a K(x) dv_0\ri|\le C
  \ee
  for $\a=1, 2, 3$ and for all $r>0$.
  \end{enumerate}
  Then there is a smooth positive function $u$ such that near infinity $u(x)=1+\frac{m}{2|x|}+w(x)$
  with $w(x)=O_4(|x|^{-2})$ and the conformally flat asymptotically Schwarzschild metric
  $$
  g_{ij}=u^4\delta_{ij}
  $$
  has scalar curvature $R(x)=a K(x)$ for some positive constant $a$.
\end{thm}
\begin{proof} By \cite[Theorem 1.4]{Ni-1982}, there is a smooth positive solution $u$ of
$$
8\Delta u+Ku^5=0
$$
such that $b\le u\le \frac1b$ for some constant $b>0$ and $u$ tends to a constant near infinity.  Let $v$ be the solution of
$$
 \Delta v=-\frac18Ku^5
$$
as in Lemma \ref{l-existence-1}.
Since $|K(x)u^5(x)|\le C|x|^{-4}$ for some constant $C$ for all $x$, we have $v(x)=O(|x|^{-1})$, see Remark \ref{rem-decay}. So $u-v$ is bounded harmonic function and must be constant. Hence  $u=c+v$ for some positive constant $c$. Replacing, $u$ by $u/c$, still denoted by $u$, we see that
$$
8\Delta u+aKu^5=0
$$
for some $a>0$. Moreover, $u=1+v$, where $v$ is the solution of
$8\Delta v=-aKu^5$ constructed in Lemma \ref{l-existence-1}.

We will estimate the derivatives of $u$.  Using the fact that $u$ is bounded and condition (i) and using cutoff function argument, one can conclude that $\int_{B_x(1)}|\nabla u|^2dv_0\le C_1$ for some constant independent of $x$. Now $u_i=\frac{\p u}{\p x^i}$ satisfies:
$$
\Delta u_i+5u^4K u_i+u^5\frac{\p K}{\p x_i}=0.
$$
By the mean value inequality \cite[Theorm 8.15]{GilbargTrudinger}, we conclude that $|\nabla u|$ is uniformly bounded. By interior Schauder estimates and (i), one can conclude that $|\p  u| =O(|x|^{-1})$, $|\p\p u| =O(|x|^{-2})$, $|\p\p\p u|=O(|x|^{-3})$.  Hence $u^5K=O_3(|x|^{-4})$.

 On the other hand, we know that $u(x)=1+O(|x|^{-1})$ by Remark \ref{rem-decay}. Since $|K(x)|\le C(|x|^{-4})$, $|x^\a(u^5(x)-1)K(x)|$ is integrable. By assumption (ii), we conclude that there is a constant $C$ such that
$$
\lf|\int_{B_0(r)}x^\a K(x)u^5(x) dv_0\ri|\le C
$$
for all $r>0$ for $\a=1,2,3$. By Lemma \ref{l-existence-1}, we conclude that
$$
u(x)=1+v(x)=1+\frac{m}{2|x|}+w(x)
$$
with $w(x)=O_4(|x|^{-2})$. Hence $g_{ij}=u^4\delta_{ij}$  is asymptotically Schwarzschild with scalar curvature $R(x)=aK(x)$ for some $a>0$.
\end{proof}
Combining this with Theorem \ref{thm-CS}, we have:

\begin{cor}\label{cor-example} Suppose $K\ge0$ is a smooth function on $\R^3$ satisfying the following: \begin{enumerate}
  \item [(i)] $K(x)=O_3( |x|^{-4})$.
  \item [(ii)]  There is a constant $C>0$ such that
  \be
  \lf|\int_{B(r)}x^\a K(x) dv_0\ri|\le C
  \ee
  for $\a=1, 2, 3$ and for all $r>0$.
  \item[(iii)] For some $\a=1, 2, 3$,
  $
  \lim_{r\to\infty} \int_{B(r)}x^\a K(x) dv_0
  $  does not exist.
  \end{enumerate}
Let $u$ be the positive function obtained in Theorem \ref{t-existence}. Then the asymptotically Schwarzschild metric $g_{ij}=u^4\delta_{ij}$ will have nonnegative scalar curvature so that the Corvino-Schoen center of mass and hence the Husiken-Yau center of mass does not exist.
\end{cor}
Hence in order to find examples of conformally flat asymptotically Schwarzschild metric on $\R^3$ with nonnegative scalar curvature so that the Corvino-Schoen center of mass and hence the Husiken-Yau center of mass does not exist, it is sufficient to find $K(x)$ satisfying the assumptions of the Corollary.
\vskip .3cm

{\bf Example}: Let us construct $K(x)$ satisfying the conditions in the Corollary. Let $\phi$ be a smooth function depending only on $r=|x|$ such that $\phi(x)=\frac{\cos(\log |x|)}{|x|^5}$ outside $|x|\ge 1$, say. Let $\eta$ be another positive function depending only on $r=|x|$ such that $\eta(x)=\frac{1}{|x|^4}$ outside $r\ge1$. Let $ \mathbf{b}$ be a nonzero vector in $\R^3$. Then one can find a positive constant $A$ such that
$$
K(x)=\phi(x)\mathbf{b}\cdot x  +A\eta(x)
$$
is positive on $\R^3$. Then one can check that $K$ satisfies (i) in Corollary \ref{cor-example}. For $\a=1, 2, $ or 3, and for any $r>0$,
$$
\int_{B(r)}x^\a  \eta(x)dv_0=0
$$
because $\eta$ depends only on $r$. Hence  for $r$ large enough
\bee
\begin{split}
\int_{B(r)}x^\a K(x)dv_0=&\int_{B(r)}x^\a\phi(x)\mathbf{b}\cdot x dv_0\\
&\int_{B(1)}x^\a\phi(x)\mathbf{b}\cdot x dv_0+b^\a\int_1^r\frac{\cos(\log t)}{t^5}\lf(\int_{S(t)}
 (x^\a)^2 d\sigma_t\ri)dt\\
 =&\int_{B(1)}x^\a\phi(x)\mathbf{b}\cdot x dv_0+b^\a \int_1^r \frac{\cos(\log t)}{t^5}\cdot \frac{4\pi t^4}{3}dt\\
 =&\int_{B(1)}x^\a\phi(x)\mathbf{b}\cdot x dv_0+\frac{4\pi b^\a}{3}\sin(\log r)
\end{split}
\eee
where $S(t)=\{x\in \R^3|\ |x|=t\}$ and $d\sigma_t$ is the area element of $S(t)$. Hence $K$  satisfies (ii) and  (iii) in Corollary \ref{cor-example}.

\vskip .3cm

{\sl Acknowledgement}: The authors would like to thank Justin Corvino, Lan-Hsuan Huang, Mu-Tao Wang and Rugang Ye for useful discussions. We would also like to thank Carla Cederbaum and Christopher Nerz for bringing our attention to their paper \cite{CederbaumNerz2013}.

\end{document}